\documentclass[12pt]{article}
\usepackage{amsmath, amsfonts, amsthm, amssymb, verbatim}
\usepackage{graphicx}
\usepackage[mathcal]{euscript}
\usepackage{amstext}
\usepackage{amssymb,mathrsfs}
\usepackage{bbm}
\usepackage[latin1]{inputenc}
\newcommand{\R}{\mathbb R}

  \newcommand{\E}{\mathbb E}

\newcommand{\PP}{\mathbb P}

\newcommand{\calX} {\ensuremath {\mathcal{X}}}

\newcommand{\calY} {\ensuremath {\mathcal{Y}}}
\newcommand{\calR} {\ensuremath {\mathcal{R}}}
\newcommand{\calF} {\ensuremath {\mathcal{F}}}

\newcommand{\vp}{\varphi}

\newcommand \loc    {\text{loc}}

\newcommand{\dive}{{\rm div}}

\newtheorem{theorem}{Theorem}[section]
\newtheorem{proposition}[theorem]{Proposition}
 \newtheorem{remark}[theorem]{Remark}
\newtheorem{lemma}[theorem]{Lemma}

\newtheorem{definition}[theorem]{Definition}
\newtheorem{hypothesis}[theorem]{Hypothesis}

\begin{document}
\title{Stochastic transport equations
\\
with unbounded divergence}

\author{Wladimir Neves\footnote{Instituto de Matem\'atica, Universidade Federal
do Rio de Janeiro, Brazil. 
E-mail: {\sl wladimir@im.ufrj.br}.
}, Christian Olivera(corresponding author)\footnote{Departamento de Matem\'{a}tica, Universidade Estadual de Campinas, Brazil. 
E-mail:  {\sl  colivera@unicamp.br}.
}}

\date{}

\maketitle


\vspace{0.3cm} \noindent {\bf MSC2010 subject classification:} 60H15, %SPDE
 35R60, %SPDE
 35F10, %Initial conditions for 1-order lin DE
 60H30. %applications of Stoch. An. 

%%%%%%%%%%%%%%%%%%%%%%%%%%%%%%
\begin{abstract}
We study in this article the existence and uniqueness of solutions to a class of 
stochastic transport equations with irregular coefficients and unbounded divergence.
In the first result we assume  the drift is $L^{2}([0,T] \times \R^{d})\cap L^{\infty}([0,T] \times \R^{d})$ and the  divergence is the locally integrable. In the  second result  we show that the smoothing  acts as a selection criterion when the  drift  is in $L^{2}([0,T] \times \R^{d})\cap L^{\infty}([0,T] \times \R^{d})$ without any condition on the divergence. 
 \end{abstract}
%%%%%%%%%%%%%%%%%%%%%%%%%%%%%%

\maketitle

%\begin{center}
%{\large
%Christian Olivera\footnote{Research  supported  FAEPEX 1324/12, FAPESP 2012/18739-0, 2012/18780-0  . }}\\
%
%\textit{Departamento de
% Matem\'{a}tica, Universidade Estadual de Campinas, \\ 
% Campinas - SP, Brasil}
%\\  e-mail:  colivera@imeunicamp.br}
%\end{center}
%%%%%%%%%%%%%%%%%%%%%%%%%%%%%%%%%%%%%%%%%%%%%%%%%%%%
\section {Introduction} \label{Intro}
%%%%%%%%%%%%%%%%%%%%%%%%%%%%%%%%%%%%%%%%%%%%%%%%%%%%

\subsection{Background} 

The linear transport equation
\begin{equation}\label{trasports}
    \partial_t u(t, x) +  b(t,x) \cdot  \nabla u(t,x)  = 0
\end{equation}
appears in a broad spectrum of physical applications, for instance related to fluid dynamics
as it is well described in Lions' books \cite{lion1,lion2}. Moreover, 
Dafermos' book \cite{Dafermos} presents more general applications of the transport equations
in the domain of conservation laws.

\medskip
The existence and uniqueness of classical solutions for the
transport equation \eqref{trasports}, with smooth coefficients
are well known. Indeed, for the so called drift $b \in  L^{\infty} ([0,T]\times \R^d )$
and having Lipschitz bounds, one applies the method of characteristics 
and the classical solution is given by $u(t, x):= u_0(\mathbf{X}^{-1}(t,x))$, where 
the flow $\mathbf{X}$ is the solution, for each $t \in (0,T)$, of the following
system of differential equations
\begin{equation}
\label{CPODE}
   \left \{
   \begin{aligned}
   &\frac{d \mathbf{X}(t)}{dt}= b(t, \mathbf{X}(t)),
   \\[5pt]
   & \mathbf{X}(0)= x.
   \end{aligned}
      \right.
\end{equation}

The above approach used to be the natural one, that is the passage from 
the Lagrangean into Eulerian formulation. The first reversed approach
was considered in 1989 by R. DiPerna, P.L. Lions \cite{DL}.
Indeed,
in that paper they proved that $W^{1,1}$ 
spatial regularity of $b(t,x)$
(together with a condition of boundedness on the divergence) is enough to ensure uniqueness of 
(renormalized) weak solutions
for the transport equations. One recalls that, the existence of bounded solutions is easily 
obtained by standard regularization of the coefficients (and passage to the limit), but
the uniqueness problem is much more delicate. 
Then, they deduced the existence, uniqueness and stability results for \eqref{CPODE}
from corresponding results on the associated linear transport equation. 
L. Ambrosio in \cite{ambrisio} extended the theory developed in  \cite{DL} to include $BV$ vector fields.
At the same time, Ambrosio introduced a probabilistic non trivial axiomatization based on the duality 
between flows and continuity equation. More recently, S. Bianchini, P. Bonicatto in \cite{BianBoni}
proved the uniqueness for \eqref{trasports} in the nearly incompressible BV
vector field setting,
%(without any assumption on the divergence), 
hence positively establishing Bressan's compactness conjecture, see \cite{BRESSANCONJ}.
Furthermore, 
we address the readers to two excellent summaries in 
\cite{ambrisio2} and \cite{lellis2}. 

\subsection{Purpose of this paper}

In the current contribution, we are interested 
in the extension of the theory developed for \eqref{trasports} under 
random perturbations of the drift vector field, 
namely considering the  stochastic linear transport equation  
%(SLTE)
%
\begin{equation}
\label{trasportUNO}
% \left \{
%\begin{aligned}
    \partial_t u(t, x) + \big(b(t, x) + \frac{d B_{t}}{dt}\big ) \cdot  \nabla u(t, x) = 0 \, ,
%    \\[5pt]
 %   &u|_{t=0}=  u_{0} \, ,
%\end{aligned}
%\right .
\end{equation}
when the divergence of the drift $b$ is not bounded, and even integrable. These assumptions
are mathematically interesting and also fundamental for some physical applications.
Moreover, the uniqueness results are false under this assumptions for the corresponding 
deterministic equation \eqref{trasports}, see for instance \cite{DL} Section IV.1. 

\medskip
 Let us recall that, the theory for  the  stochastic linear transport equation  \eqref{trasportUNO} 
has been developing quite well. Indeed,
the first and influential result in this direction was 
given by F. Flandoli, M. Gubinelli, E. Priola
\cite{FGP2}. In that paper, they obtained wellposedness of the stochastic problem for an H\"older 
continuous drift term (with some integrability conditions on the divergence), where their 
approach is based on a careful analysis of the characteristics.
Since the paper \cite{FGP2}, there exist a considerable list of important correlated results, 
to mention a few  \cite{AttFl11}, \cite{Beck}, \cite{Fre1}, \cite{FGP2}, 
\cite{Gali}, \cite{Gesssmith} \cite{Moli}, \cite{NO}, \cite{Ol}, \cite{Wei}. For related
 works see \cite{Alonso}, \cite{Alonso2}, \cite{Gess3}, \cite{Gess}, \cite{Gess4} and \cite{O2}.

\medskip
It is well known that, when divergence of $b$ is not bounded we may develop 
concentrations or vacuum in the continuity equation, which is related to the transport
equation, at least formally. Indeed, let us consider the continuity equation (deterministic case)
$$
    \partial_t v(t, x) +  b(t,x) \cdot  \nabla v(t,x)= - v(t,x) \, {\rm div}b(t,x). 
$$
Applying the simple change of variables $u= v \, e^{\int^t {\rm div} b}$, 
we get the linear transport equation \eqref{trasports}. 
Therefore, boundeness of ${\rm div} b$ prevents some physical extreme 
cases, but which are interesting for applications. 
For instance, the transport-continuity equation with unbounded divergence
 is an important application to hyperbolic systems that, can be written as 
one scalar conservation law plus one transport-continuity equation, 
where the divergence of  the drift term is a measure, 
see \cite{Spino}. 

\medskip
Here, we consider the stochastic linear transport equations
such that  ${\rm div} b$ is not bounded, first, and then, even integrable.
Hence we have to deal with a more general sense of solutions, 
and in the second case the uniqueness property gives up its place to a 
selection principle (see next section).  

%%%%%%%%%%%%%%%%%
\subsection{Selection principle}
%%%%%%%%%%%%%%%%%

It is well known that many partial differential equations cannot, 
in general, have a classical solution, at least for several
important 
%physical 
applications. Then, one may
adopt the concept of weak solutions (for PDEs in divergence form)
or for instance, the concept of viscosity solutions as introduced by 
P.L. Lions, M. G. Crandall \cite{Crandall}. 
Unfortunately, the concepts of nonclassical 
solutions permit nonuniqueness. Therefore, it is natural to 
inquiry if there is a selection principle (or admissibility criteria),
which would from the physical viewpoint select the correct 
physical solution among all nonclassical ones, and ensure the 
uniqueness of a nonclassical solution from the mathematical
viewpoint. 

\medskip
There exist some strategies to tackle the above problem.
One of them is the concept of entropy solutions for 
conservation laws, where an adicional equation is considered, in fact an inequality, which
is automatically satisfied for classical solutions, but play the role of 
a selection principle for weak ones. The existence of entropies 
for the scalar case is trivial, but more delicate for systems of
conservation laws, see \cite{DSSB}. Similarly, we address the 
reader to the concept of maximal dissipation introduced by C.M. Dafermos \cite{Dafermos},
and well adapted by E. Feireisl \cite{EF1} to play the role of a selection principle,
which is violated by the oscillatory solutions
(obtained in the process of convex integration) as introduced
by C. DeLellis, L. Sz\'ekelyhidi \cite{CamiloLaslo}. 
In addition, we address the reader to
some recent results of non-uniqueness for (deterministic) transport equations
obtained also via convex integration, see S. Modena \cite{Modena} 
(and references therein). 

\medskip
Some types of approximations are also used to guarantee uniqueness at the approximate level, and then
establishing an admissible criteria. For instance, the  zero-viscosity limit works well to seek for
(admissible) viscosity solutions of Hamilton-Jacobi equations. 
Furthermore, the weak solutions for scalar conservation laws 
that are uniquely selected by entropy selection principle,  
coincide with solutions obtained by the vanishing viscosity method, 
see C. Dafermos \cite{Dafermos1}. Albeit, the zero-viscosity limit of weak solutions of
Navier Stokes equations, which are not Leray type, does not select a unique solution
of the Euler equations, see T. Buckmaster, V. Vicol \cite{TBVV}. Therefore, 
the vanishing viscosity limit is still poorly mathematically understood
at this time.

%%%%%%%%%%%%%%%%%%%%%%%%%%%%%%
\subsection{The contribution of this paper is two-fold} 
%%%%%%%%%%%%%%%%%%%%%%%%%%%%%%

The first result establishes 
the existence and uniqueness of the equation \eqref{trasportUNO} on the class of quasi-regular solution, 
assuming  that 
$b\in L^{2}([0,T]\times \R^{d})\cap L^{\infty}([0,T]\times \R^{d})$
and ${\rm div}\, b\in L_{loc}^{1}([0,T]\times \R^{d})$. In the  second result,
we show that the smoothing  acts as a selection criterion when 
$b\in L^{2}([0,T]\times \R^{d})\cap L^{\infty}([0,T]\times \R^{d})$, that is, 
the equation with regularized  drift term converges to only one week-limit point 
in $L_{loc}^{2}(\Omega \times [0,T]\times \R^{d})$. The precise definitions
and statements 
will be presented in Section \ref{STATEMAIN}. 

\medskip
One remarks that, both results cited above are intrinsically stochastic. 
Then, we list some others
known conditions on the divergence, related to this paper, in other to  
obtain uniqueness of solution to transport-continuity equations. 
In \cite{Mucha}
is showed uniqueness when the divergence is the $BMO$ space. In \cite{Clop} it is assumed some  global exponential integrability conditions  (both results for the 
deterministic equation). In \cite{Beck} is showed uniqueness of the stochastic transport-continuity equation when 
the divergence is in $L_{t}^{q}(L_{x}^{p})$ with $d/p + 2/q< 1$.
 
\bigskip 
This paper is organized as follows. In the next section we present precisely the setting, introduce some notation, define the class 
of quasiregular weak solutions, and week-asymptotic solutions. Also we state the main results of this paper. 
In Section \ref{PROOFS}, we present the proof of the main results.
To ease the presentation, the proofs of some technical results are postponed to the Appendix.

%%%%%%%%%%%%%%%%%%%%%%%%%%%%%%%%%%%%%%
%\section{Definition of  weak solutions}
\section{Statements of the Main Results} \label{STATEMAIN}
%%%%%%%%%%%%%%%%%%%%%%%%%%%%%%%%%%%%%%

The aim of this section is to present the setting and 
suitable definitions of weak solutions to equation \eqref{trasportUNO}, adapted to treat the problem of well-posedness 
under our very weak assumptions on the regularity of the coefficients and the initial condition.
Moreover, we state the main results of this paper. 

\begin{hypothesis}\label{hyp} 
We assume the following conditions for the vector field $b$,
\begin{equation}\label{con1}
   b \in L^2\big( [0,T] \times \R^d  \big) \cap  L^\infty\big( [0,T] \times \R^d  \big),
\end{equation} 
\begin{equation}\label{con2}
 \dive \, b \in L_{loc}^1 ( [0,T] \times  \R^d ) \, . 
\end{equation}
Moreover, the initial condition is taken to be
\begin{equation}\label{conIC}
 u_0 \in  L^2(\R^d)  \cap L^\infty (\R^d) \, . 
\end{equation}
\end{hypothesis}

\medskip
We shall work on a fixed time interval $t\in[0,T]$, and throughout the paper 
we will use a given probability space $(\Omega, \PP, \calF)$, on which there exists 
an $\R^d$-valued Brownian motion $B_t$ for $t\in[0,T]$. We will use the 
natural filtration of the Brownian motion $\calF_t = \calF_t^B$, 
and restrict ourselves to consider the collection of measurable sets given by the $\sigma$-algebra $\calF = \calF_T$, augmented by the $\PP$-negligible sets. 
Moreover, for convenience 
we introduce the following set of random variables, called the space of stochastic exponentials
$$ 
  \calX := \Big\{ F= \exp \Big( \int_0^T h(s) \cdot d B_s - \frac{1}{2}\int_0^T |h(s)|^2 \, ds \Big)   \    \Big|     \    h \in L^2\big([0,T] ;\R^d \big) \Big\} \, .
$$
Further details on stochastic exponentials and some useful properties are collected in the Appendix.
In particular,
the technical assumption that the $\sigma$-algebra we are using is 
the one provided by the Brownian motion is essential to ensure that 
the family of stochastic exponentials provides a set of test functions 
large enough to obtain almost sure uniqueness.

\medskip
The next definition tells us in which sense a stochastic process is a weak solution of \eqref{trasportUNO}. 
Hereupon, we will use the summation convention on repeated indices.

\begin{definition}\label{defisolu}  Under conditions \eqref{con1}, \eqref{con2} and \eqref{conIC},
a stochastic process $u\in  L^\infty \big( \Omega \times[0,T] \times \R^d \big)$ is called 
a quasiregular weak solution of the Cauchy problem \eqref{trasportUNO}, when
\begin{itemize}
\item ({\it Weak solution}) For any test function $\varphi \in C_c^{\infty}(\R^d)$, the real valued process $\int  u(t,
  x)\varphi(x)  dx$ has a continuous modification which is an
$\mathcal{F}_{t}$-semimartingale, and for all $t \in [0,T]$, we have $\mathbb{P}$-almost surely
\begin{equation} \label{DISTINTSTR}
\begin{aligned}
    \int_{\R^d} u(t,x) \varphi(x) dx = &\int_{\R^d} u_{0}(x) \varphi(x) \ dx
   +\int_{0}^{t} \!\! \int_{\R^d} u(s,x) \ b^i(s,x) \partial_{i} \varphi(x) \ dx ds
\\[5pt]
    &     +   \int_{0}^{t} \!\! \int_{\R^d} \dive b(s,x) u(s,x) \varphi(x) \ dx \, ds 
\\[5pt]    
    &+ \int_{0}^{t} \!\! \int_{\R^d} u(s,x) \ \partial_{i} \varphi(x) \ dx \, {\circ}{dB^i_s} \, .
\end{aligned}
\end{equation}

\item ({\it Regularity in Mean}) For each function $F \in \calX$, the deterministic function $V:=\mathbb{E}[uF]$ is a measurable bounded function, 
which belongs to  $L^{2}([0,T]; H^{1}(\R^d) ) \cap C([0,T]; L^2(\R^d))$.
\end{itemize}
\end{definition}

\medskip
One remarks that, the stochastic
integration in this paper is to be understood in the Stratonovich sense.
This is often considered to be the natural one for this kind of problems, that is,
in view of the Wong-Zakai approximation theorem. Moreover, 
it is useful for computations to present also the Ito formulation of equation \eqref{DISTINTSTR}. It reads 
\begin{align}\label{Ito-weak}
    \int_{\R^d} u(t,x) \varphi(x) \, dx =  &  \int_{\R^d} u_0(x) \varphi(x) \, dx    \\
    & + \int_{0}^{t} \!\! \int_{\R^d} u(s,x) \, \big( b(s,x) \cdot \nabla \varphi(x) + \varphi(x)   \dive b(s,x) \big) \, dx ds     \nonumber   \\
    & + \int_{0}^{t}  \Big( \int_{\R^d} u(s,x) \ \nabla \varphi(x) \ dx \Big) \cdot dB_s  \nonumber  \\
    & + \frac{1}{2}  \int_{0}^{t} \!\! \int_{\R^d} u(s,x) \ \Delta \varphi(x) \, dx ds \, .   \nonumber
\end{align}

Then, we are read to state our first main result. 
\begin{proposition}\label{lemmaexis1}  
Under the conditions of Hypothesis \ref{hyp}, there exist quasi-regular weak solutions of the Cauchy problem
for the linear stochastic transport equation
\eqref{trasportUNO}. 
\end{proposition}

The key hypothesis to prove existence of quasiregular solutions, under Hypothesis \ref{hyp},
 is condition \eqref{con1} which allows to obtain a-priori estimates. 
The existence result is somehow classical (see for example \cite{FGP2} ), 
but we still have to check the regularity in mean of such solutions.

\medskip
For the uniqueness, let us consider the following {\it family} of parabolic equations,
that is to say 
\begin{equation}\label{para}
\partial_t V(t,x) + \big( b(t,x) + h(t) \big) \cdot \nabla V(t,x)= \frac{1}{2} \Delta V(t,x)
\end{equation}
with $h(t)\in L^2(0,T)$.   We will show (see Lemma \ref{lem:eq V} below) that for a quasiregular weak solution $u$ 
to the stochastic transport equation \eqref{trasportUNO}, its expected value $V=\E[uF]$ against any stochastic 
exponential $F$ solves, as soon as it is sufficiently regular, a parabolic equation of the family \eqref{para}, 
and therefore, as one could then expect, is unique.  Using the uniqueness result not only for a single 
equation but for the whole family \eqref{para}, and looking at stochastic exponentials as test functions 
(they form a family which is large enough), we are able to obtain almost sure uniqueness.  This idea was introduced 
in \cite{Fre2}, the  new ingredient in this work  is to show uniqueness with unbounded  divergence.
The following lemma is proved in \cite{Fre2} and after that, we state the uniqueness result. 
\begin{lemma}\label{lem:eq V}
If $u$ is a quasiregular weak solution of \eqref{trasportUNO}, then for each function $F \in \calX$, the deterministic function $V:=\mathbb{E}[uF]$ satisfies  the parabolic equation \eqref{para} in the weak sense, with initial condition given by $V_0 = u_0$.
\end{lemma}

\begin{proof}
Take any $F\in \calX $ and any quasiregular weak solution $u$. By definition, $V(t,x) \in L^{2} \big( [0,T]; H^{1}(\R^d) \big) \cap C \big( [0,T] ; L^2 (\R^d) \big)$. Consider the It\^o integral form of the equation satisfied by $u$, as given in \eqref{Ito-weak}. To obtain an equation for $V$ we multiply this equation by $F$ and take expectations:
\begin{align}\label{eq1}
    \int_{\R^d} V(t,x) \varphi(x) \, dx = & \int_{\R^d} V(0,x) \varphi(x) \, dx \nonumber \\
    & + \int_{0}^{t} \!\! \int_{\R^d} V(s,x) \, \big( b(s,x) \cdot \nabla \varphi(x) + \varphi(x)   \dive b(s,x) \big) \, dx ds      \nonumber  \\[5pt]
    & + \E \Big[  \int_{0}^{t}  \Big( \int_{\R^d} u(s,x) \ \nabla \varphi(x) \ dx \Big) \cdot dB_s  \ F \Big] \\[5pt]  
    & + \frac{1}{2}  \int_{0}^{t} \!\! \int_{\R^d} V(s,x) \ \Delta \varphi(x) \, dx ds \, .   \nonumber
\end{align} 
By definition of quasiregular weak solutions, $\int_{\R^d} u(\cdot,x) \varphi(x) \, dx$ is an adapted square integrable process for any $\varphi\in C^\infty_c(\R^d)$. Therefore,
$$Y_s = \int_{\R^d} u(s,x) \nabla \varphi(x) \, dx$$
is also an adapted square integrable process. The expected value of the stochastic integral on the third line of \eqref{eq1} can be rewritten as the expected value of a Lebesgue integral against a certain function $h \in L^2\big([0,T] \big)$ due to the properties of stochastic exponentials:
\begin{align*}
 \E \Big[  \int_{0}^{t}  \Big( \int_{\R^d} u(s,x) \ \nabla \varphi(x) \ dx \Big) \cdot dB_s  \ F \Big] =  \int_{0}^{t} \int_{\R^d} \!  V(s,x) h(s) \cdot \nabla \varphi(x) \, dx ds   \, .
\end{align*}
 This is shown in detail in Lemma \ref{lemma B-F} in the Appendix.
\medskip

Now, due to the regularity of $V$, we see that $V$ is a weak solution of the parabolic equation \eqref{para}, that is to say, for each test function $\vp \in C^\infty_c(\R^d)$ 
\begin{align}\label{eq V}
    \int_{\R^d} V(t,x)  \varphi (x) \, dx =  & \int_{\R^d} V(0,x)  \varphi (x) \, dx   \nonumber  \\
     &+  \int_{0}^{t} \int_{\R^d} \!  V(s,x) \big(   b(s,x) \cdot  \nabla \varphi (x) +   \varphi (x) \dive b(s,x)   \big)  \, dx  ds     \nonumber    \\[5pt]
     &+  \int_{0}^{t} \int_{\R^d} \!  V(s,x) h(s) \cdot \nabla \varphi(x) \, dx ds     \nonumber    \\[5pt]
     & - \frac{1}{2}\int_{0}^{t}   \int_{\R^d} \!   \nabla V(s,x)  \cdot \nabla \varphi (x) \, dx  ds  \, .
\end{align}

As explained in the Appendix, due to the properties of stochastic exponentials we have that $F$ is a martingale with mean $1$. Since $u_0$ is deterministic, it immediately follows that $V_0 = \E \big[ u_0 F \big] = u_0$.
\end{proof}

\medskip
Then, we can state the uniqueness result.
\begin{theorem}\label{uni} 
Under the conditions of Hypothesis \ref{hyp}, uniqueness holds for quasiregular weak solutions of  the Cauchy problem 
for the linear stochastic transport equation \eqref{trasportUNO} in the following sense:
if $u,v \in  L^\infty \big(\Omega \times [0,T] \times \R^d \big)$ are two quasiregular weak solutions 
with the same initial data $u_{0}$, then  $u= v$ almost everywhere 
in $ \Omega  \times [0,T] \times \R^d $. 
\end{theorem}

\medskip
Now, we consider the second main result of this paper. 

\medskip
To this end, we say that a random field $\{ S(t,x) : t \in [0,T],~x \in \mathbb{R} \}$ is a spatially dependent semimartingale if for
each $x \in \mathbb{R}$, $\{S(t,x): t \in [0,T]\}$ is a semimartingale in relation to the same filtration $\{ \mathcal{F}_t : t \in [0,T] \}$. If
$S(t,x)$ is a $C^{\infty}$-function of $x$ and continuous in $t$ almost everywhere,  it is called  a $C^{\infty}$-semimartingale.  
See \cite{Ku} for a rigorous study of spatially depend semimartingales
and applications to  stochastic differential equations.

\medskip
In the following, we introduce a new concept of solution for the stochastic transport  equation \eqref{trasportUNO}. 
\begin{definition} We say that $u$ in $L^{2}(\Omega\times [0,T]; L_{loc}^{2}(\R^{d}))\cap  L^{\infty}(\Omega\times [0,T] \times \R^{d})$ is a 
week-asymptotic solution of the Cauchy problem 
for the linear stochastic transport equation \eqref{trasportUNO}, if
\begin{enumerate}
\item There exists a sequence of $C^{\infty}$-semimartingales $\{ u_{\epsilon} \}_{\epsilon> 0}$, such that
$ u=\lim_{\epsilon \rightarrow 0}u_{\epsilon}$   
weak in $L^{2}(\Omega\times [0,T], L_{loc}^{2}(\R^{d})) $ and $\ast-$weak  in  $L^{\infty}(\Omega\times [0,T] \times \R^{d})$.

\item  For all $\epsilon>0$, the semimartingale $u_{\epsilon}$ verifies 
\[
 u^\varepsilon (t, x,\omega) =u_{0}^{\epsilon}(x) + \int_{0}^{t}  \nabla u^\varepsilon (s, x,\omega)  \cdot b^\varepsilon (s, x)  ds +
+    \int_{0}^{t}   \nabla u^\varepsilon  \circ d B_{t}(\omega), 
\]
where $u_{0}^{\epsilon}$ and $b^\varepsilon $ are mollfiers approximation of $u_{0}$ and $b$ respectively.
\end{enumerate}
\end{definition}

Then, we have the following selection principle.
\begin{proposition} 
\label{PropSP} 
Under conditions \eqref{con1} and \eqref{conIC}, 
%let  $\{u^\varepsilon\}$  be an arbitrary family of solutions of the regularized problem \eqref{trasport-reg}. Then 
there exists a unique week-asymptotic solution of the Cauchy problem
for the linear stochastic transport equation \eqref{trasportUNO}.
\end{proposition}

%%%%%%%%%%%%%%%
\section{Strategy of Proofs}
\label{PROOFS}
%%%%%%%%%%%%%%%

%%%%%%%%%%%%%%%%%%%%%%%%%%%%%%%%%
\subsection{ Existence Proof}
\label{EXISTENCE}
%%%%%%%%%%%%%%%%%%%%%%%%%%%%%%%%%

The main issue of this section is to prove Proposition \ref{lemmaexis1}. 

\begin{proof}[{\bf Proof of Proposition \ref{lemmaexis1}}] We divide the proof into two steps. First, using an approximation procedure we shall
 prove that the problem   \eqref{trasportUNO}
 admits weak solutions under our hypothesis. Then, in the second step, we show that the solutions obtained as limit of regularized problems in the first step are
   indeed quasiregular solutions.
\bigskip

{\it Step 1: Weak solution property.} Let $\{\rho_\varepsilon\}_{\varepsilon> 0}$ be a family of standard symmetric mollifiers with compact support.
 Using this family  of functions we define the family of regularized coefficients as
 $b^{\varepsilon}(t,x) =   (b(t,\cdot) \ast \rho_\varepsilon (\cdot) ) (x)  $.
  Similarly, define the family of regular approximations of the initial condition 
	$u_0^\varepsilon (x) = (u_0(\cdot) \ast \rho_\varepsilon (\cdot) )(x)  $.

One remark that, any element $b^{\varepsilon}$, $u_0^\varepsilon$ of the two families
 we have defined is smooth (in space) and compactly supported, hence with bounded
 derivatives of all orders. Then, for any fixed $\varepsilon>0$, the classical theory of Kunita, see \cite{Ku} or \cite{Ku3},
 provides the existence of a unique 
 solution $u^{\varepsilon}$ to the regularized equation 
\begin{equation}\label{trasport-reg}
 \left \{
\begin{aligned}
    &d u^\varepsilon (t, x,\omega) +  \nabla u^\varepsilon (t, x,\omega)  \cdot \big( b^\varepsilon (t, x)  dt +
 \circ d B_{t}(\omega) \big) = 0 \, ,
    \\[5pt]
    &u^\varepsilon \big|_{t=0}=  u_{0}^\varepsilon
\end{aligned}
\right .
\end{equation}
together with the representation formula
\begin{equation}\label{repr formula}
u^\varepsilon (t,x) = u_0^\varepsilon \big( (\phi_t^\varepsilon)^{-1} (x) \big)
\end{equation}
in terms of the (regularized) initial condition and the inverse flow $(\phi_t^\varepsilon)^{-1}$ associated to the equation of characteristics of \eqref{trasport-reg}, which reads 
\begin{equation*}
d X_t = b^\varepsilon (t, X_t) \, dt + d B_t \, ,  \hspace{1cm}   X_0 = x \,.
\end{equation*}

\medskip

If $u^\varepsilon$ is a solution of \eqref{trasport-reg}, it is also a weak solution, 
which means that for any test function $\varphi \in C_c^\infty(\R^d)$, 
$u^\varepsilon$ satisfies for all $t \in [0,T]$,  
the following equation (written in It\^o form)
\begin{equation} \label{transintegralR2}
\begin{aligned}
    \int_{\R^d} u^\varepsilon(t,x) &\varphi(x) \, dx= \int_{\R^d} u^\varepsilon_{0}(x) \varphi(x) \, dx
   +\int_{0}^{t} \!\! \int_{\R^d} u^\varepsilon(s,x) \, b^\varepsilon(s,x) \cdot \nabla \varphi(x) \, dx ds
\\[5pt]
   &+ \int_{0}^{t} \!\! \int_{\R^d} u^\varepsilon(s,x)\,  \dive \, b^\varepsilon(s,x) \, \varphi(x) \, dx ds 
\\[5pt]
    &+ \int_{0}^{t} \!\! \int_{\R^d} u^\varepsilon(s,x) \, \partial_{i} \varphi(x) \, dx \, dB^i_s
    + \frac{1}{2} \int_{0}^{t} \!\!\int_{\R^d} u^\varepsilon(s,x) \Delta \varphi(x) \, dx ds \, .
\end{aligned}
\end{equation}
To prove the existence of weak solutions to \eqref{trasportUNO} we shall show that the sequence $u^\varepsilon$ admits a convergent subsequence, and pass to the limit in the above equation along this subsequence. This is done following the classical argument of \cite[Sect. II, Chapter 3]{Pardoux}, see also \cite[Theorem 15]{FGP2}.

\medskip
 By the representation formula \eqref{repr formula} itself, we also get the uniform bound in 
$L^\infty  \big( \Omega \times [0,T] \times \R^d \big)$. Therefore, there exists a sequence 
$\varepsilon_n \to 0$ such that $u^{\varepsilon_n}$ weak-$\star$ converges in $L^\infty$ 
and weakly in $L^2(\Omega \times [0,T], L_{loc}^{2}(\R^{d}))$ to some process 
$u\in L^\infty  \big( \Omega \times [0,T] \times \R^d \big) \cap  L^2(\Omega \times [0,T], L_{loc}^{2}(\R^{d})) $. To ease notation, let us denote $\varepsilon_n$ by $\varepsilon$ and for every $\varphi\in  C_c^{\infty}(\R^d)$, $\int_{\R^d} u^\varepsilon (t,x) \varphi(x) \, dx $ by $u^\varepsilon(\varphi)$, including the case $\varepsilon=0$.

\medskip
Clearly, along the convergent subsequence found above, the sequence of nonanticipative processes $u^{\varepsilon}(\varphi)$ also weakly converges in $L^2 \big( \Omega \times [0,T])$ to the process $u( \varphi)$, which is progressively measurable because the space of non\-anticipative processes is a closed subspace of $L^2 \big( \Omega \times [0,T])$, hence weakly closed. It follows that the It\^o integral of the bounded process $u(\varphi)$ is well defined. Moreover, the mapping $f\mapsto \int_0^\cdot f(s) \cdot dB_s$ is linear continuous from the space of nonanticipative $L^2(\Omega \times[0,T] ; \R^d)$-processes to $L^2(\Omega \times [0,T])$, hence weakly continuous. Therefore, the It\^o term $\int_0^\cdot u^\varepsilon(\nabla \varphi) \cdot \, dB_s$ in \eqref{transintegralR2} converges weakly in $L^2(\Omega\times [0,T])$ to $\int_0^\cdot u(\nabla \varphi) \cdot \, dB_s$.

\medskip
Note that the coefficients $b^\varepsilon$ and $\dive b^\varepsilon$ are strongly convergent in $L^2([0,T] \times \R^d)$ and $L_{loc}^1\big([0,T] \times\R^d \big)$ respectively. This implies that $b^\varepsilon \cdot \nabla \varphi + \varphi \dive b^\varepsilon$ strongly converges in $L^1([0,T]; L^1(\R^d) )$ to $b \cdot \nabla \varphi + \varphi \dive b$ because $\varphi$ is of compact support. We can therefore pass to the limit also in all the remaining terms in \eqref{transintegralR2}, to find that the limit process $u$ is a weak solution of \eqref{trasportUNO}.

\bigskip
{\it Step 2: Regularity.}  First,
let us denote by $\calY$ the separable metric 
space $C([0,T];  L^{2}(\R^d))$, and
consider a solution $u_\varepsilon$ of the 
regularized problem \eqref{trasport-reg}. For any $F\in \calX$, the function $V_{\varepsilon}(t,x):= \mathbb{E} [u^{\varepsilon}  (t,x) F ]$
is regular and we can apply Lemma \ref{lem:eq V} to get
$$
\begin{aligned}
    V_{\varepsilon}(t,x)  &=  V_{0}^\varepsilon(x)
   -\int_{0}^{t} \! \nabla V_{\varepsilon}(s,x) \cdot \big( {b^\varepsilon}(s,x) + h(s) \big) \, ds
     + \frac{1}{2}\int_{0}^{t} \!    \Delta V_{\varepsilon}(s,x)   \, ds \,.
\end{aligned}
$$
Rewrite this in differential form:
\begin{align*}
\partial_t V_{\varepsilon}^2(t,x)  &=    - \nabla V_{\varepsilon}^2(s,x) \cdot \big( {b^\varepsilon}(s,x) + h(s) \big)  + 
 V_\varepsilon \Delta V_{\varepsilon}(s,x)    \,.
\end{align*}
Now, integrating in time and space we get
$$
\begin{aligned}
    \int_{\R^d} V_{\varepsilon}^{2}(t,x)   \, dx &=  \,  \int_{\R^d} \big(V_{0}^\varepsilon \big)^{2}(x) \, dx    \\[5pt]
     &  -  2 \int_{0}^{t} \int_{\R^d} \! \nabla V_{\varepsilon}(s,x) V_{\varepsilon}(s,x)  b^\varepsilon(s,x)  \, dx  ds \\[5pt]
		     &  -  \int_{0}^{t} \int_{\R^d} \! \nabla \big( V_{\varepsilon}^{2} \big) (s,x) \cdot  h(s)  \, dx  ds \\[5pt]
     & 	- \int_{0}^{t}   \int_{\R^d} \!    \big| \nabla V_{\varepsilon}(s,x) \big| ^{2} \, dx  ds \, ,
\end{aligned}
$$
and rearranging the terms conveniently we obtain the bound
\begin{equation}
\label{Gronwall1}
\begin{aligned}
    \int_{\R^d} V_{\varepsilon}^{2}(t,x)   \, dx & 
     + \int_{0}^{t}   \int_{\R^d} \!    \big| \nabla V_{\varepsilon}(s,x) \big| ^{2} \, dx  ds    
    \leq  \int_{\R^d} \big(V_{0}^\varepsilon \big)^{2}(x) \, dx    \\[5pt]
     &  +  2 \int_{0}^{t} \int_{\R^d} \! |\nabla V_{\varepsilon}(s,x) \, V_{\varepsilon}(s,x) \, b^\varepsilon(s,x)|  \, dx  ds \\[5pt]
     &   \le \int_{\R^d} (V_{0}^\varepsilon)^{2}(x) \, dx
     \\[5pt]
    &  +   C  \int_{0}^{t}  \int_{\R^d} 
		|V_{\varepsilon}(s,x)|^{2} \, dx  ds  +   \frac{1}{4}  \int_{0}^{t}  \int_{\R^d} 
		|\nabla V_{\varepsilon}(s,x)|^{2} \, dx  ds \, ,
\end{aligned}
\end{equation}
where the positive constant $C$ can be chosen uniformly in $\varepsilon$.
Then,  we can apply Gr\"onwall's Lemma to obtain
\begin{equation}   \label{un}
    \int_{\R^d} V_{\varepsilon}^{2}(t,x)  \, dx  \leq C
			\int_{\R^d} (V_{0}^\varepsilon)^{2}(x) \, dx,
\end{equation}
and plugging \eqref{un} into \eqref{Gronwall1}, we also get
\begin{equation}\label{dos}
\int_{0}^{t}   \int_{\R^d} \!    \big| \nabla V_{\varepsilon}(s,x)  \big| ^{2}  \, dx  ds   \leq   C   \int_{\R^d} \big( V_{0}^\varepsilon \big)^{2}(x) \, dx \ .
\end{equation}

From (\ref{un}) and (\ref{dos}) we deduce the existence of a subsequence $\varepsilon_n$ 
(which can be extracted from the subsequence used in the previous step) for which  $ V_{\varepsilon_n}(t,x)$ converges weakly to the function $V(t,x)=\E[u(t,x)\, F]$ in $ \calY$ and such that $ \nabla V_{n}(t,x)$ converges weakly to $\nabla V(t,x)$ in $ L^{2}([0,T] \times \R^d )$. This allows us 
 to conclude that $V\in L^2([0,T]; H^{1}(\R^d) ) \cap C([0,T]; L^2(\R^d))$. Moreover, since $u$ is a bounded function,
this carries over to $V$. 
  \end{proof}

%%%%%%%%%%%%%%%%%%
\subsection{Uniqueness Proof} 
\label{UNIQUE}
%%%%%%%%%%%%%%%%%%

In this section, we shall prove Theorem \ref{uni}.
%for the SPDE \eqref{trasportUNO}, where
%As in the by now classical setting, 
The proof relies on the commutator Lemma \ref{conmuting}, below. 
If applied in the usual way, this lemma requires to have $W^{1,1}$ regularity either for the drift coefficient 
$b$ or for the solution $u$. This is precisely what we want to avoid: in our setting we have neither of them, 
since we want to deal with possibly discontinuous solutions and drift coefficients. However, 
the key observation is that, it is enough to ask such Sobolev regularity for the expected values 
$V(t,x)= \E [u(t,x) F]$ for $F\in \calX$, and not on the solution $u$ itself. 

\medskip
Before stating and proving the main theorem of this section, we shall introduce some further notations and the key lemma on commutators. We stress that in this section we will be working under both the sets of Hypothesis \ref{hyp}.

\bigskip
Let  $\{\rho_{\varepsilon} \}$ be a family of standard positive symmetric mollifiers. Given two functions $f:\R^d \mapsto \R^d$ and $g:\R^d \mapsto \R$, the commutator $\calR_\varepsilon(f,g)$ is defined as
\begin{equation}\label{def commut}
    \mathcal{R}_{\varepsilon}(f,g):= (f \cdot \nabla ) (\rho_{\varepsilon}\ast g )- \rho_{\varepsilon}\ast  (f\cdot \nabla g ) \, .
  \end{equation}
The following lemma is due to Le Bris and Lions \cite{BrisLion}.

\begin{lemma} \label{conmuting} (C. Le Bris - P. L.Lions )
Let  $f \in  L^2_\loc(\R^d ) $,  $g \in H^{1}(\R^d) $. Then, 
passing to the limit as $\varepsilon\rightarrow 0$
\[
    \mathcal{R}_{\varepsilon}(f,g) \rightarrow 0  \qquad in  \qquad L^1_\loc(\R^d) \, .
  \]
\end{lemma}

\begin{proof}[{\bf Proof of Theorem \ref{uni}}] 
The proof relies on energy-type estimates on $V$ (see equation \eqref{eqen} below) combined with Gr\"onwall's Lemma. 
However, to rigorously obtain \eqref{eqen} two preliminary technical steps of regularization and localization 
are needed, where the above Lemma \ref{conmuting} will be used to deal with the commutators appearing 
in the regularization process.  \medskip

{\it Step 0: Set of solutions.} Remark that the set of quasiregular weak solutions is a linear subspace of $L^\infty \big(\Omega \times [0,T] \times \R^d \big)$, because the stochastic transport equation is linear, and the regularity conditions is a linear constraint. Therefore, it is enough to show that a quasiregular weak
solution $u$ with initial condition $u_0= 0$ vanishes identically. \bigskip

{\it Step 1: Smoothing.}
Let $\{\rho_{\varepsilon}(x)\}_\varepsilon$ be a family of standard symmetric mollifiers. For any $\varepsilon>0$ and $x\in\R^d$ we can use $\rho_\varepsilon(x-\cdot)$ as test function in the equation \eqref{eq V} for $V$. Observe that considering only quasiregular weak solutions starting from $u_0=0$ results in $V_0=0$. Using the regularity of $V$, we get
$$
\begin{aligned}
      \int_{\R^d} V(t,y) \rho_\varepsilon(x-y) \, dy  = &\, - \int_{0}^{t}  \int_{\R^d} \big( b(s,y) \cdot \nabla V(s,y)  \big)  \rho_\varepsilon(x-y) \ dy ds    \\[5pt]
        &-  \int_{0}^{t} \int_{\R^d} \! \big( h(s) \cdot \nabla V(s,x) \big) \rho_\varepsilon(x-y) \, dy ds     \\[5pt]
    & - \frac{1}{2}\int_{0}^{t}   \int_{\R^d} \! \nabla V(s,y) \cdot  \nabla_y \, \rho_\varepsilon(x-y) \, dy  ds \,.
\end{aligned}
$$
For each $t \in [0,T]$, we set $V_\varepsilon(t,x)= V(t,x) \ast \rho_\varepsilon(x)$, and using the definition \eqref{def commut} of the commutator 
$\big(\calR_\varepsilon(f,g)\big) (s)$ with $f=b(s, \cdot)$ and $g=V(s, \cdot)$, we have
$$
\begin{aligned}
    V_{\varepsilon}(t,x) + \int_{0}^{t} \big( b(s,x) + h(s) \big) \cdot  \nabla V_{\varepsilon}(s,x) \,  ds   & -    \frac{1}{2}\int_{0}^{t}   \Delta V_{\varepsilon}(s,x) \, ds      \\[5pt]
    & =         \int_{0}^{t} \big(\mathcal{R}_{\varepsilon}(b,V) \big) (s) \,  ds  \, .
\end{aligned}
$$
Due to $b$ and $V$ regularities, given by \eqref{con1}, and the solution Definition \ref{defisolu}, one easily obtains that $\mathcal{R}_{\varepsilon}(b,V)\in L^1\big( [0,T] ; L^1_{loc} (\R^d) \big)$. 
 Therefore, $V_\varepsilon$ is differentiable in time. To obtain an equation for $V_\varepsilon^2$ we can differentiate the above equation in time, multiply by $2V_\varepsilon$ and integrate again. We end up with
\begin{equation}
\label{1000}
\begin{aligned}
    V_{\varepsilon}^{2}(t,x) + \int_{0}^{t} \big( b(s,x) + h(s) \big)  \cdot  \nabla \big( V_{\varepsilon}^{2}(s,x) \big) \, ds
-  \int_{0}^{t}    V_\varepsilon (s,x)  \Delta V_{\varepsilon} (s,x) \, ds
\\[5pt]
	=	 2 \int_{0}^{t}  V_{\varepsilon}(s,x)   \mathcal{R}_{\varepsilon}(b,V) \, ds  \, .
\end{aligned}
\end{equation}
Remark that, by definition of solution, $V$ is bounded. Therefore, $V_\varepsilon$ 
is uniformly bounded. It follows that all the terms above have the right integrability properties, and the equation is well-defined.

\bigskip
{\it Step 2: Energy inequality and Passage to the limit.} 
Integrating equation \eqref{1000} in space, we have 
$$
\begin{aligned}
    \int_{\R^d}   V_{\varepsilon}^{2}(t,x)   \, dx \,  + & \int_{0}^{t}   \int_{\R^d}  |\nabla V_{\varepsilon}(s,x)|^{2}   \, dx  ds
\\[5pt]
	&\leq 2 \int_{0}^{t}  \int_{\R^d} |V_{\varepsilon}(s,x) \, \nabla V_{\varepsilon}(s,x) \cdot  b(s,x) |  \,  dx   ds
	\\[5pt]
	&	\quad  +  2 \int_{0}^{t}  \int_{\R^d} | V_{\varepsilon}(s,x)  \,  \mathcal{R}_{\varepsilon}(b,V)|  \,dx  ds.
\end{aligned}
$$
Then we obtain 
\begin{equation}
\label{ener0}
\begin{aligned}
    \int_{\R^d}   V_{\varepsilon}^{2}(t,x)   \, dx \,  
    + & \frac{1}{2}\int_{0}^{t}   \int_{\R^d}  |\nabla V_{\varepsilon}(s,x)|^{2}   \, dx  ds
\\[5pt]
	&\leq  C \int_{0}^{t}  \int_{\R^d}   | V_{\varepsilon}(s,x)|^{2}   \,  dx   ds
	\\[5pt]
	&	\quad  +  2 \int_{0}^{t}   \int_{\R^d}   |V_{\varepsilon}(s,x) \,  \mathcal{R}_{\varepsilon}(b,V)| \, dx ds,
\end{aligned}
\end{equation}
where the positive constant $C$ can be chosen uniformly in $\varepsilon$. 
Recall that $u$ is bounded, so that $V$ and $V_\varepsilon$ are (uniformly) bounded too. 
Moreover, by standard properties of mollifiers, it follows that  $V_\varepsilon \to V$ strongly in $L^{2}\big([0,T]; H^1(\R^d)\big) \cap C([0,T] ; L^2 (\R^d) )$, and we can use
 Lemma \ref{conmuting} and the uniform boundedness of $V_\varepsilon$ to deal with the
  term on the right hand side. Then, passing to the limit as $\varepsilon\rightarrow 0$ in the above equation 
\eqref{ener0}, we get
\begin{equation}\label{eqen}
\begin{aligned}
    \int_{\R^d}   V^{2}(t,x)   \, dx \,  + & \frac{1}{2}\int_{0}^{t}   \int_{\R^d}  |\nabla V(s,x)|^{2}   \, dx  ds
\\[5pt]
	&\leq  C \int_{0}^{t}  \int_{\R^d}   | V(s,x)|^{2}   \,  dx   ds. 
\end{aligned}
\end{equation}
 Applying Gr\"onwall's Lemma we conclude that for every $t\in[0,T]$, $V(t,x)=\E[u(t,x) F]=0$ for almost every $x\in\R^d$ and every $F\in \calX$. 

\bigskip
{\it Step 4: Conclusion.} From the result of the previous step we get that $\int_{[0,T]\times\R^d} \E [u(t,x) F] f(t,x)\, dx dt =0$ for all $F\in \calX$ and $f\in C^\infty_c ([0,T] \times \R^d)$. By linearity of the integral and the expected value we also have that
\begin{equation}\label{eq uY}
\int_{[0,T] \times \R^d} \E\big[ u(t,x) \,Y \big] f(t,x) \, dx dt =0
\end{equation}
for every random variable $Y$ which can be written as a linear combination of a finite number of $F\in \calX$. By Lemma \ref{expo} the span generated by $\calX$ is dense in $L^2(\Omega)$, 
hence \eqref{eq uY} holds for any $Y\in L^2(\Omega)$. Linear combinations of products of functions $Y f(t,x)$ are dense in the space 
of test functions $\psi(\omega,x,t) \in L^2(\Omega  \times [0,T]\times \R^d)$, so that for any compact set $K \subset \R^{d}$
\begin{equation*}
\int_{[0,T]\times K} \E \big[ u(t,x) \,\psi(\omega,t,x) \big]  \, dxdt =0  \, .
\end{equation*}
Consequently,  $u=0$ almost everywhere on $\Omega\times [0,T] \times \R^d$. 
\end{proof}

\begin{remark} 
\label{REMAS}  From the proof of Theorem \ref{lemmaexis1} it is possible to see that, under our weak hypothesis, {\bf any} weak solution $u$ of the Cauchy problem \eqref{trasportUNO} which is the $L^\infty\big( \Omega ; \calY \big)$-limit of weak solutions to regularized problems has the regularity of a quasiregular 
weak solution, and is therefore unique by Theorem \ref{uni}. 
In other words, we have also proved uniqueness in the sense of Theorem \ref{uni} 
in the class of solutions which are limit of regularized problems.
\end{remark}

%%%%%%%%%%%%%%%%%%%%%%%%%%%%%%%%%%%%%
\subsection{Selection principle Proof}
%%%%%%%%%%%%%%%%%%%%%%%%%%%%%%%%%%%%%

Finally, we present the 
\begin{proof}[{\bf Proof of Proposition \ref{PropSP}}]   Let $\{u^\varepsilon\}$, $\{u^\delta\}$ 
be two families of solutions of the regularized problem \eqref{trasport-reg} corresponding
respectively to mollifiers $\{\rho_{\varepsilon}\}_\varepsilon$ and $\{\rho_{\delta}\}_\delta$.
Appling the same procedure done in the proof of the Proposition (\ref{lemmaexis1}),
there are  sequences  
$\varepsilon_n \to 0$  and $\delta_n \to 0$ such that $u^{\varepsilon_n}, u^{\delta_n}$ weak-$\star$ converges in $L^\infty(\Omega \times [0,T]\times \R^{d})$ 
and weakly in $L^2(\Omega \times [0,T], L_{loc}^{2}(\R^{d}))$ to some processes  
$u, v \in L^\infty  \big( \Omega \times [0,T] \times \R^d \big) \cap  L^2(\Omega \times [0,T], L_{loc}^{2}(\R^{d})) $. Moreover, 
$V_{\epsilon_{n}}= \E[u^{\epsilon_{n}} F]$,  and $W_{\delta_{n}}= \E[v^{\delta_{n}} F]$ converge to $V$ and $W$ respectively in 
$L^2([0,T]; H^{1}(\R^d) ) \cap C([0,T]; L^2(\R^d))$. To ease notation, let us denote $\varepsilon_n$ by $\varepsilon$
and $\delta_n$ by $\delta$.

Now, we will show that $V=W$.  We  observe that $U_{\epsilon, \delta}=V_{\epsilon}- W_{\delta}$ verifies 
$$
    U_{\varepsilon, \delta}(t,x)  =  (V_{0}^\varepsilon-W_{0}^{\delta})(x)
   -\int_{0}^{t} \! \nabla U_{\varepsilon, \delta}(s,x) \cdot \big( {b^\varepsilon}(s,x) + h(s) \big) \, ds
$$	
	$$
	 -\int_{0}^{t} \! \nabla W_{\delta}(s,x) \cdot  ( b^\varepsilon(s,x) -b^\delta(s,x))  \, ds
     + \frac{1}{2}\int_{0}^{t} \!    \Delta U_{\varepsilon, \delta}(s,x)   \, ds \,.
$$

Then we have 
$$
\begin{aligned}
\partial_t U_{\varepsilon, \delta}^2(t,x) &=    - \nabla U_{\varepsilon, \delta}^2(t,x) \cdot \big( {b^\varepsilon}(s,x) + h(s) \big) 
\\[5pt]
&+  U_{\varepsilon, \delta}(t,x) \  \Delta U_{\varepsilon,\delta}(t,x)  + 2  U_{\varepsilon,\delta}(t,x) \nabla W_{ \delta}(t,x) \cdot  (b^\varepsilon-b^{\delta} )(t,x).
\end{aligned}
$$
	
Now, integrating in time and space we obtain
$$
\begin{aligned}
    \int_{\R^d} U_{\varepsilon,\delta}^{2}(t,x)   \, dx& = \,  \int_{\R^d} \big(V_{0}^\varepsilon-W_{0}^\delta \big)^{2}(x) \, dx    \\[5pt]
     &  -  2 \int_{0}^{t} \int_{\R^d} U_{\varepsilon, \delta}(s,x) \, \nabla U_{\varepsilon,\delta}(s,x) \cdot  (b^\varepsilon(s,x) + h(s) )  \, dx  ds \\[5pt]
		     &  - 2 \int_{0}^{t} \int_{\R^d} U_{\varepsilon, \delta}(s,x) \,  \nabla W_{\delta}(s,x) \cdot (b^\varepsilon(s,x)-b^\delta(s,x ))          \, dx  ds \\[5pt]
     & 	- \int_{0}^{t}   \int_{\R^d} \!    \big| \nabla U_{\varepsilon,\delta}(s,x) \big| ^{2} \, dx  ds \, ,
\end{aligned}
$$
and rearranging the terms conveniently, we have
$$
\begin{aligned}
    \int_{\R^d} U_{\varepsilon,\delta}^{2}(t,x)   \, dx& + \int_{0}^{t}   \int_{\R^d} \!    \big| \nabla U_{\varepsilon,\delta}(s,x) \big| ^{2} \, dx  ds\leq \int_{\R^d} \big(V_{0}^\varepsilon-W_{0}^\delta \big)^{2}(x) \, dx    
    \\[5pt]
     &  +  2 \int_{0}^{t} \int_{\R^d} |U_{\varepsilon, \delta}(s,x)| \, |\nabla U_{\varepsilon,\delta}(s,x) \cdot  b^\varepsilon(s,x)  |  \, dx  ds 
     \\[5pt]
		     &  + 2 \int_{0}^{t} \int_{\R^d} |U_{\varepsilon, \delta}(s,x)| \,  |\nabla W_{\delta}(s,x) \cdot (b^\varepsilon(s,x)-b^\delta(s,x ))|         \, dx  ds.
\end{aligned}
$$
Then, applying H\"older's inequality 
$$
\begin{aligned}
    &\int_{\R^d} U_{\varepsilon,\delta}^{2}(t,x)   \, dx + \int_{0}^{t}   \int_{\R^d} \!    \big| \nabla U_{\varepsilon,\delta}(s,x) \big| ^{2} \, dx  ds\leq \int_{\R^d} \big(V_{0}^\varepsilon-W_{0}^\delta \big)^{2}(x) \, dx    
    \\[5pt]
     &  +  C \int_{0}^{t} \int_{\R^d} |U_{\varepsilon, \delta}(s,x)|^2 \, dx  ds +  \frac{1}{4}  \int_{0}^{t}  \int_{\R^d} 
		|\nabla U_{\varepsilon}(s,x)|^{2} \, dx  ds
     \\[5pt]
     & + C (\int_{0}^{t} \int_{\R^d}  |\nabla W_{\delta}(s,x)|^{2} \ dx  ds)^{1/2}  (\int_{0}^{t} \int_{\R^d} |b^\varepsilon(s,x)-b^\delta(s,x)|^2 \, dx  ds)^{1/2},
%		     &  + 2 \int_{0}^{t} \int_{\R^d} |U_{\varepsilon, \delta}(s,x)| \,  |\nabla W_{\delta}(s,x) \cdot (b^\varepsilon(s,x)-b^\delta(s,x ))|         \, dx  ds.
\end{aligned}
$$		
where the positive constant $C$ does not depend on $\varepsilon,\delta> 0$. 
Due to Gronwall's inequality,  we have 
\begin{equation}\label{Gronwall2}
\begin{aligned}
    &\int_{\R^d} U_{\varepsilon,\delta}^{2}(t,x)   \, dx \leq \int_{\R^d} \big(V_{0}^\varepsilon-W_{0}^\delta \big)^{2}(x) \, dx    
%    \\[5pt]
%     &  +  C \int_{0}^{t} \int_{\R^d} |U_{\varepsilon, \delta}(s,x)|^2 \, dx  ds +  \frac{1}{4}  \int_{0}^{t}  \int_{\R^d} 
%		|\nabla U_{\varepsilon}(s,x)|^{2} \, dx  ds
     \\[5pt]
     & + C (\int_{0}^{t} \int_{\R^d}  |\nabla W_{\delta}(s,x)|^{2} \ dx  ds)^{1/2}  (\int_{0}^{t} \int_{\R^d} |b^\varepsilon(s,x)-b^\delta(s,x)|^2 \, dx  ds)^{1/2}.
%		     &  + 2 \int_{0}^{t} \int_{\R^d} |U_{\varepsilon, \delta}(s,x)| \,  |\nabla W_{\delta}(s,x) \cdot (b^\varepsilon(s,x)-b^\delta(s,x ))|         \, dx  ds.
\end{aligned}
\end{equation}	
Passing to the limit as $n \to \infty$, we have $U= 0$, that is, $V= W$, 
and arguing as in the step 4 of the Proposition (\ref{uni}),
we conclude that $u=v$. 

\medskip
Finally, we have that all sequence has subsequences that converge to the same limit, 
which implies our result. 
\end{proof}

\begin{remark} A  similar concept of week-asymptotic solution 
have been considered by other authors  and   has proved to be an efficient mathematical tool to
 study explicitly creation and superposition of singular solutions to various nonlinear PDEs, see for instance 
\cite{Albe}, \cite{Albe2}, \cite{Colom2}, \cite{Danilov}, \cite{Danilov2}, \cite{Panov}. 
\end{remark}

\appendix
%%%%%%%%%%%%%%%%%%%%%%%%
\section{Appendix}
%%%%%%%%%%%%%%%%%%%%%%%%

\begin{definition}\label{def Ft}
Given a filtered probability space with an $\R^d$-valued Brownian motion defined on it, 
$( \Omega, \calF, P, \calF_t, B_t)$, for any $h \in L^2([0,T] ;\R^d)$, we can define the random process
$$
 F_t= \exp \Big( \int_0^t h(s) \cdot d B_s - \frac{1}{2} \int_0^t  |h(s)|^2 \, ds \Big) \, ,
$$
for $t\in[0,T]$. Such random processes are called stochastic exponentials.
\end{definition}

\medskip
We recall that stochastic exponentials satisfy the following SDE ( see \cite[proof of Theorem 4.3.3]{oksen} ) 
\begin{equation}\label{SDEexpon2}
F_t =   1 +  \int_{0}^{t}  h(s) F_s \, dB_s \, .
\end{equation}
This can be obtained by applying It\^o formula to $F_t$. By Novikov's condition it also follows that any stochastic exponential $F_t$ is an $\calF_t$-martingale, and $\E[F_t] = 1$.

\medskip
When $t=T$, we shall use the short notation $F=F_T$ and, with a slight abuse of notation, 
still call the random variable $F$ a stochastic exponential. Let us recall the definition of the following space of random variables, which we call the space of stochastic exponentials:
$$ 
  \calX := \Big\{ F= \exp \Big( \int_0^T h(s) \cdot d B_s - \frac{1}{2} \int_0^T |h(s)|^2 \, ds \Big)   \    \Big|     \    h \in L^2 \big( [0,T] ;\R^d \big) \Big\} \, .
$$

\begin{remark}
Even though it is not really essential for our proof, we point out that for every $F \in \calX$ there exists a unique $h\in L^2(0,T)$ such that $F$ is the stochastic exponential of $h$. This can be easily shown using It\^o isometry.
\end{remark}

The following result, see \cite[Lemma 4.3.2]{oksen} or   
\cite[Lemma 2.3]{Mao}, is a key fact for our analysis. Recall that $\calF=\calF_T$.

\begin{lemma}\label{expo} 
The span generated by $ \calX $ is a dense subset of $ L^{2}( \Omega)$.
\end{lemma}

We also have the following result.

\begin{lemma}\label{lemma B-F}
Let $F$ be a stochastic exponential and $Y_s\in L^2\big(\Omega\times [0,T]\big)$ an $\R^d$-valued, square integrable adapted process. Then,
\begin{equation}\label{B-F}
\E \Big[\int_{0}^{t}  Y_s \cdot d B_s \ F  \Big] =  \int_0^t h(s) \cdot \E \big[ Y_s  \ F   \big]  \, ds \, .
\end{equation}
\end{lemma}

\begin{proof}
Using the representation formula \eqref{SDEexpon2} we have
\begin{align*}
\E \Big[\int_{0}^{t}  Y_s \cdot d B_s \ F  \Big] &= \E \Big[\int_{0}^{t}  Y_s \cdot d B_s  \Big] +  \E \Big[\int_{0}^{t}  Y_s \cdot d B_s \ \int_0^T h(s) F_s \cdot dB_s  \Big] \\
&=  \E \Big[\int_{0}^{t}  Y_s \cdot h(s) F_s \, ds  \Big] \, .
\end{align*}
Since that $Y_s$ is $\mathcal{F}_{s}$-adapted, we obtain  
$$
\E \big[ Y_s   \, F_s  \big]  =  \E \big[ Y_s   \, F  \big] \, ,
$$
and \eqref{B-F} follows.
\end{proof}

%%%%%%%%%%%%%%%%%%%%%%%%
\section*{Acknowledgements}
%%%%%%%%%%%%%%%%%%%%%%%%
Conflict of Interest: Author Wladimir Neves has received research grants from CNPq
through the grant  308064/2019-4, and also by FAPERJ 
(Cientista do Nosso Estado) through the grant E-26/201.139/2021. 
Author Christian Olivera is partially supported by  CNPq
through the grant 426747/2018-6 and FAPESP by the grants 2020/15691-2	 
and 2020/04426-6.

%%%%%%%%%%%%%%%%%


\begin{thebibliography}{9999}
%%%%%%%%%%%%%%%%%%

\bibitem{Albe}
S. Albeverio, V. Danilov, {\it Construction to global in time solution to Kolmogorov-Feller pseudodifferential equations}
with a small parameter using characteristics. Math. Nachr. 285, 426-439, 2012.

\bibitem{Albe2}
 S. Albeverio, O. Rozanova,{\it A representation of solutions to a scalar conservation law in several dimension}, J. Math.
Anal. Appl. 405, 711-719, 2013. 


\bibitem{Alonso}
D.  Alonso-Oran,  A. Bethencourt de Leon,  S. Takao,   
{\it  The Burgers equation with stochastic transport: shock formation, local and global existence of smooth solutions},  
Nonlinear Differential Equations and Applications NoDEA 26,  57, 2019. 

\bibitem{Alonso2}
D.  Alonso-Oran,  A. Bethencourt de Leon,  S  {\it  On the Well-Posedness of Stochastic Boussinesq Equations with Transport Noise},
Journal of Nonlinear Science,  30, 175-224, 2020. 

\bibitem{ambrisio}
L.  Ambrosio, {\it Transport equation and
Cauchy problem for $BV$ vector fields}, Invent. Math.,  158,  227--260, 2004.


\bibitem{ambrisio2}
L. Ambrosio, G. Crippa. (2014) {\it Continuity equations and ODE
fows with non-smooth velocity},  Lecture Notes of a course given at HeriottWatt University, Edinburgh. Proceeding
of the Royal Society of Edinburgh, Section A: Mathematics, 144, 1191-1244.


\bibitem{AttFl11} S. Attanasio and F. Flandoli, {\it Renormalized Solutions for Stochastic Transport Equations and the Regularization by Bilinear Multiplicative Noise}. Comm. in Partial Differential Equations, 36(8), 1455--1474, 2011.


\bibitem{Beck}
L. Beck, F. Flandoli, M. Gubinelli, M. Maurelli, {\it
Stochastic ODEs and stochastic linear PDEs with critical drift: regularity, duality and uniqueness }, Electron. J. Probab.
 24, 2019.


\bibitem{DSSB} S. Benzoni-Gavage, D. Serre,
{\it Multidimensional hyperbolic partial differential equations, First order systems and applications}. 
Oxford Mathematical Monographs. The Clarendon Press Oxford University Press, Oxford, 2007.

\bibitem{BianBoni} S. Bianchini, P. Bonicatto,
{\it  A uniqueness result for the decomposition of vector fields in $\R^d$}, 
Invent. Math., 220, 225--393, 2020. 


\bibitem{BRESSANCONJ} A. Bressan,
{\it An ill posed Cauchy problem for a hyperbolic system in two space dimensions,}
Rend. Sem. Mat. Univ. Padova, 110, 2003, 103--117.


\bibitem{TBVV} T. Buckmaster, V. Vicol,
{\it Nonuniqueness of weak solutions to the Navier-Stokes equation}. 
Ann. Math., 189, (2019), 101--144.

\bibitem{Clop} A.  Clop1 · R. Jiang, J.  Mateu1, J.  Orobitg {\it
Linear transport equations for vector fields with subexponentially integrable divergence},
Calc. Var,   55-21, 2021. 

\bibitem{Colom2}
M. Colombeau, {\it Radon measures as solutions of the Cauchy problem for evolution equations}, 
Z. Angew. Math. Phys. 71, 2020. 


\bibitem{Crandall}
M.G. Crandall, P.L Lions, 
{\it  Viscosity solutions of Hamilton-Jacobi equations}.  
Trans. Amer. Math. Soc. 277, 1--42, 1983. 

\bibitem{Spino}
G. Crippa 1,  Laura V. Spinolo {\it An overview on some results concerning the transport equation and its applications to conservation laws}, Communications on Pure  Applied Analysis, 9, 1283-1293, 2010. 




\bibitem{Dafermos1} C.M. Dafermos, 
{\it Hyperbolic conservation laws in continuum physics}. 
Third edition. Grundlehren der Mathematischen Wissenschaften 
[Fundamental Principles of Mathematical Sciences], 325. Springer-Verlag, 2010.

\bibitem{Gess3}
K. Chouk, B.  Gess {\it Path-by-path regularization by noise for scalar conservation laws}, 
Journal of Functional Analysis, 277,  1,   1469-1498, 2019. 
\bibitem{Dafermos}
	C.M. Dafermos, Hyperbolic conservation laws in continuum physics. Third edition. 
	Grundlehren der Mathematischen Wissenschaften [Fundamental Principles of Mathematical Sciences], 325. Springer-Verlag, 2010.

\bibitem{Danilov}
 V. Danilov, D. Mitrovic, {\it Delta shock wave formation in the case of triangular hyperbolic system of conservation laws}, 
 J. Differ. Equ. 245, 3704-3734, 2008. 


\bibitem{Danilov2}
 V. Danilov, V.Shelkovich, {\it Delta shock wave formation in the case of triangular hyperbolic system of conservation laws}, 
 J. Differ. Equ., 211, 333-381, 2005. 



\bibitem{lellis2}
C. De Lellis. \emph{Notes on hyperbolic systems of conservation laws and transport equations.}
In Handbook of differential equations: evolutionary equations. Vol. III, 
Handb. Differ. Equ., pages 277--382. Elsevier/North-Holland, Amsterdam, 2007.

\bibitem{CamiloLaslo} C. De Lellis, L. Sz\'ekelyhidi, 
{\it On admissibility criteria for weak solutions of the Euler equations}. 
Arch. Ration. Mech. Anal., 195 (1 ), 225--260, 2010.




\bibitem{DL}
R. DiPerna and P.L. Lions, {\it Ordinary differential
equations, transport theory and Sobolev spaces}. Invent. Math.,  98, 
511--547, 1989.

\bibitem{Fre1}
 E. Fedrizzi and F. Flandoli, {\it Noise prevents singularities in linear transport equations}.  
 Journal of Functional Analysis, 264,  1329--1354, 2013.

\bibitem{Fre2} E. Fedrizzi, W. Neves, C. Olivera, 
\textsl{\ On a class of stochastic transport equations for $L_{loc}^{2}$ vector fields}.  
Annali della Scuola Normale Superiore di Pisa, Classe di Scienze, 18,  397-419, 2018.

\bibitem{EF1} E. Feireisl,
{\it Maximal dissipation and well-posedness for the compressible Euler system}. 
Arch. Ration. Mech. Anal., 16, 447--461, 2014.


\bibitem{FlanLuo}
 F. Flandoli and D. Luo. {\it High mode transport noise improves vorticity blowup
control in 3d navier-stokes equations},  arXiv preprint: 1910.05742, 2019.


\bibitem{Flanlect}
F. Flandoli, Random perturbation of PDEs and fluid dynamic models. 
Lectures from the 40th Probability Summer School held in Saint-Flour, 2010.
 Lecture Notes in Mathematics, 2015. Springer, Heidelberg, 2011. 



\bibitem{FGP2}
 F. Flandoli, M. Gubinelli and E. Priola, {\it Well-posedness of the transport equation by stochastic  perturbation}. Invent. Math., 180, 1--53, 2010.

\bibitem{Gali}
L.  Galimberti, K. H Karlsen, {\it Renormalization of stochastic continuity equations on Riemannian manifolds },   
arXiv:1912.10731, 2019. 





\bibitem{Gess}
B. Gess and M. Maurelli, {\it Well-posedness by noise for scalar conservation laws}, 43, Communications in Partial Differential Equations
1702-1736, 2019.


\bibitem{Gess4}
B. Gess and P.E. Souganidis, {\it Long-Time Behavior, Invariant Measures, and Regularizing 
Effects for Stochastic Scalar Conservation Laws}, Communications on Pure and Applied Mathematics, 70, 8, 1562-1597, 2017.

\bibitem{Gesssmith}
B. Gess,   S.Smith {\it Stochastic continuity equations with conservative noise}, Journal de Mathematiques Pures et Appliquees
 128, 225--263, 2019. 

 


\bibitem{Ku}
 H. Kunita, Stochastic flows and stochastic differential
equations. Cambridge University Press, 1990.

\bibitem{Ku3}
H. Kunita, {\it First order stochastic partial differential equations}.  In: Stochastic Analysis,
Katata Kyoto, North-Holland Math. Library, 32, 249--269, 1984.


\bibitem{BrisLion}
C. Le Bris and P.L. Lions, {\it Existence and uniqueness of solutions
to Fokker-Planck type equations with irregular coefficients}. Comm.
Partial Differential Equations, 33, 1272--1317, 2008. 

\bibitem{lion1}
P.L. Lions,  Mathematical topics in fluid mechanics, Vol. I: incompressible models. Oxford
Lecture Series in Mathematics and its applications, 3 (1996), Oxford University Press.

\bibitem{lion2}
P.L. Lions,  Mathematical topics in fluid mechanics, Vol. II: compressible models. Oxford
Lecture Series in Mathematics and its applications, 10 (1998), Oxford University Press.



\bibitem{Mao}
X. Mao, {\it  Stochastic differential equations and applications, Horwood Publishing Chichester}.
UK, 1997, second edition 2007.

\bibitem{Modena} S. Modena, 
{\it On some recent results concerning non-uniqueness for the transport equation}.
Hyperbolic problems: theory, numerics, applications. 
Proceedings of the 17th international conference on Hyperbolic Problems, 
AIMS 10, 562--568, 2020.


\bibitem{Moli}
 David A.C. Mollinedo, C. Olivera,  {\it Stochastic continuity  equation with non-smooth velocity}, 
Annali di Matematica Pura ed Applicata, 196, 1669-168, 2017. 

\bibitem{Mucha}
P. Mucha, {\it  Transport equation: Extension of classical results for $divb\in  BMO$}, 
Journal of Differential Equations,  249,  1871-1883, 2010.  


\bibitem{NO}
 W.  Neves, C. Olivera,  {\it Wellposedness for stochastic continuity equations with
 Ladyzhenskaya-Prodi-Serrin condition},   Nonlinear Differential Equations and Applications NoDEA
22, 1247-1258, 2015.



\bibitem{Ol}
  C. Olivera, {\it  Regularization by noise in one-dimensional continuity equation},  
	 Potential Anal,   51, 23-35, 2019.

\bibitem{O2}
  C. Olivera, {\it  Well-posedness of the non-local conservation law by stochastic perturbation
}, Manuscripta Mathematica,   162, 367-387, 2020. 

\bibitem{oksen}
B. \O ksendal, Stochastic Differential Equations: An Introduction With Applications.
Springer-Verlag, 2003.

\bibitem {Pardoux}E. Pardoux, \textit{Equations aux d\'{e}riv\'{e}es
partielles stochastiques non lin\'{e}aires monotones. Etude de solutions
fortes de type It\^{o}}. PhD Thesis, Universite Paris Sud, 1975.


\bibitem{oksen}
B. \O ksendal, Stochastic Differential Equations: An Introduction With Applications.
Springer-Verlag, 2003.

\bibitem{Panov}
E.  Panov, M. Shelkovich, {\it  $\delta$-shock waves as a new type of solutions to systems of conservation laws}, 
 J. Differ Equ. 228, 49-86, 2006. 

 

\bibitem{Pardoux}
E. Pardoux, {\it  Equations aux d\'eriv\'ees partielles stochastiques non lin\'eaires monotones.
Etude de solutions fortes de type It\^o}.  PhD Thesis, Universite Paris Sud, 1975.




\bibitem{Wei}
J. Wei, J. Duan, H. Gao, L.  Guangying  {\it  Stochastic regularization for transport equations
} Stochastics and Partial Differential Equations: Analysis and Computations, 9, 105-141, 2021



\end{thebibliography}
\end{document}